\documentclass[11pt]{article}
\usepackage{amsmath,amsthm,amssymb,mathabx,ifsym,a4wide} 
\usepackage{amscd, mathrsfs, wasysym, color}  
\usepackage{graphicx}
\usepackage{epsfig} 
\usepackage{slashed}  
\usepackage[hidelinks]{hyperref}
\usepackage[most]{tcolorbox} 
\newtheorem{theorem}{Theorem}[section]
\newtheorem{proposition}[theorem]{Proposition}
\newtheorem{lemma}[theorem]{Lemma}

\numberwithin{equation}{section} 
\numberwithin{figure}{section}  
\newcommand \la \langle
\newcommand \ra \rangle

\newcommand \Ecal {\mathcal{E}}

\newcommand \underdel {\underline \partial}

\newcommand \trianglerightNEW \triangleright

\newcommand \uh {\widehat u}

\newcommand \auth {\textsc}

\newcommand \Kcal {\mathcal K}

\newcommand \bei {\begin{itemize}}
\newcommand \eei {\end{itemize}}
\newcommand \be {\begin{equation}}
\newcommand \bel {\be\label}
\newcommand \ee {\end{equation}}
\newcommand \del \partial
\newcommand \RR {{\mathbb R}}

\newcommand \Hcal {\mathcal H}

\newcommand \eps \epsilon

\let\oldmarginpar\marginpar
\renewcommand\marginpar[1]{\-\oldmarginpar[\raggedleft\footnotesize #1]%
{\raggedright\footnotesize #1}}
\begin{document}

\title{\bf The zero mass problem for Klein-Gordon equations: quadratic null interactions} 

\author{Shijie Dong\footnote{
\normalsize Fudan University, School of mathematical sciences. 
\newline
Email : {\sl dongs@ljll.math.upmc.fr, shijiedong1991@hotmail.com}
}}

\date{April, 2020}

\maketitle

\begin{abstract}
We study in $\RR^{3+1}$ a system of nonlinearly coupled Klein-Gordon equations under null condition, with (possibly vanishing) mass varying in the interval $[0, 1]$. Our goal is three folds: 1) we want to establish the global well-posedness result to the system which is uniform in terms of the mass parameter; 2) we want to obtain unified pointwise decay result for the solution to the system, in the sense that the solution decays more like a wave component (independent of the mass parameter) in certain range of time, while the solution decays as a Klein-Gordon component with a factor depending on the mass parameter in the other part of the time range; 3) the solution to the Klein-Gordon system converges to the solution to the corresponding wave system in certain sense when the mass parameter goes to 0. In order to achieve these goals, we will rely on both the flat and the hyperboloidal foliation of the spacetime.
\end{abstract}

\tableofcontents
 


\section{Introduction}

\subsection{Motivation and revisit of the classical results}

The study of nonlinear wave equations, nonlinear Klein-Gordon equations, and their coupled systems has been an active area of research since decades ago, and among which the question "what kind of quadratic nonlinearities lead to global-in-time solutions" has attracted special attention from the researchers. Recall, on one hand, that wave equations in $\RR^{3+1}$ with null form nonlinearities were proved to admit global-in-time solutions independently by Klainerman \cite{Klainerman86} and Christodoulou \cite{Christodoulou}, and the generalisations \cite{Sideris96, Sideris00, L-R-cmp, L-R-annals, P-S-cpam} for instance. 
On the other hand, it was shown by Klainerman \cite{Klainerman85} and Shatah \cite{Shatah} that Klein-Gordon equations with general quadratic nonlinearities in $\RR^{3+1}$ admit small solutions. 
Besides, the global well-posedness results for different types of coupled wave and Klein-Gordon systems, with or without physical models behind, were obtained, see for instance \cite{Bachelot, DW, DLW, Georgiev, Ionescu-P, Ionescu-P2, Katayama12a, Klainerman-QW-SY, PLF-YM-book, PLF-YM-cmp, PLF-YM-arXiv1, PLF-YM-arXiv2, OTT, PsarelliA, PsarelliT, Tsutsumi, Tsutsumi2, Wang}.

We recall that (linear) Klein-Gordon components decays $t^{-1/2}$ faster than the (linear) wave components in $\RR^{n+1}$ ($n \geq 1$), and the presence of the mass term allows one to control the $L^2$--type energy of the Klein-Gordon components by their natural energies. Both of these make it less difficult to study nonlinear Klein-Gordon equations in $\RR^{3+1}$. On the other hand, we can utilise the scaling vector field which makes it easy to apply the Klainerman-Sobolev inequality, and can reply on the conformal energy estimates to obtain $L^2$--type energy estimates for wave components (with no derivatives), in studying nonlinear wave equations (but not Klein-Gordon equations). Thus we can see from the above comparisons that there are different features that help to study pure wave equations and pure Klein-Gordon equations.

Concerning the fact that Klein-Gordon equations become wave equations when the masses are set to be 0, a natural interesting question is that for the Klein-Gordon equations with varying mass in $[0, 1]$ what kind of quadratic nonlinearities can ensure the small data global existence results which are uniform in terms of the varying mass parameter. Our primary goal is prove that all kinds of the null nonlinearities can uniformly guarantee the global existence results for systems of Klein-Gordon equations with mass varying in $[0, 1]$. Such results are known to be valid at the end points $0, 1$, which are corresponding to wave equations and Klein-Gordon equations with fixed mass respectively. But more is involved if one wants get the global results uniform in terms of mass parameter in $[0, 1]$, which is due to the fact that we cannot use the scaling vector field, cannot obtain mass independent $L^2$--type estimates by the mass term or by the conformal energy estimates, ect.

In addition, the study of the Klein-Gordon equations with varying mass (especially when the mass goes to 0) is also motivated from the study of the mathematical physics. We briefly recall in \cite{DLW} that when studying the electroweak standard model there appear several Klein-Gordon equations with different masses, and physical experiments have verified that some of the masses are extremely small (close to 0 but still positive) compared to others. Thus it is important to obtain results which are uniform in terms of the small masses for Klein-Gordon equations.

\subsection{Model of interest}

We will consider the following system of coupled Klein-Gordon equations with varying mass $m \in [0, 1]$:
\bel{eq:model-mKG}
\aligned
- \Box v_i + m^2 v_i = N^{jk}_i Q_0 (v_j, v_k) + M^{jk \alpha \beta}_i Q_{\alpha \beta} (v_j, v_k),
\endaligned
\ee
with initial data prescribed on $t = t_0 = 2$
\bel{eq:mKG-ID}
\big( v_i, \del_t v_i \big)(t_0)
=
\big( v_{i0}, v_{i1} \big).
\ee
In the above, $\Box = \eta^{\alpha \beta} \del_{\alpha \beta}$ is the wave operator, where $\eta = \text{diag}(-1, 1, 1, 1)$ is the metric of the spacetime, and Einstein summation convention is adopted. The indices $i, j, k \in \{1, \cdots, N_0 \}$ with $N_0$ the number of equations (also the number of unknowns), and we use $a, b, c, \cdots \in \{1, 2, 3\}$ and $\alpha, \beta, \gamma, \cdots \{0, 1, 2, 3\}$ to denote the space indices and the spacetime indices respectively. Besides, throughout the paper we will also use $A \lesssim B$ to indicate $A \leq C B$, with $C$ a generic constant (independent of the mass parameter $m$).

At the end points of $m = 0$ and $m = 1$, the small data global existence result (as well as other properties of the solution) for system \eqref{eq:model-mKG}--\eqref{eq:mKG-ID} is well-known, and the proof depends on the features of the pure wave equations and the pure Klein-Gordon equations. Here we want to establish the global existence result and explore the properties of the solution for the system \eqref{eq:model-mKG}, which are uniform in terms of the mass parameter $m \in [0, 1]$. Besides, it is also interesting to show that the solution to \eqref{eq:model-mKG} converges to the corresponding wave system when $m \to 0$. Since some features of the pure wave equations or of the pure Klein-Gordon equations cannot be relied on in obtaining the uniform result, the analysis of the proof is more subtle, and requires new insights.

\subsection{Difficulties and new observations}

When studying the Klein-Gordon equations, the most well-known difficulty is that one cannot use the scaling vector field, which is due to the fact that the scaling vector field does not commute with the Klein-Gordon operators. However, more difficulties arise in studying the Klein-Gordon system with possibly vanishing mass. 

First, in order to apply the Sobole--type inequalities to obtain pointwise decay results for the Klein-Gordon components $v = (v_i)$ or to estimate the null forms, we need to bound the $L^2$--type norm for $v_i$, which is supposed to be mass independent. 
On one hand, the presence of the mass term in the Klein-Gordon equation does not seem to be helpful in obtaining $L^2$--type energy estimate. That is because what we can get from the mass term is only
$$
m \| v \|_{L^2} \lesssim B,
\qquad
\text{i.e. } \| v \|_{L^2} \lesssim m^{-1} B \footnote{$m^{-1}$ is interpreted as $+\infty$ when $m=0$.},
$$
which is mass dependent, but the bound for $\| v \|_{L^2}$ blows up when $m$ goes to 0,
and in the above $B$ represents some bound from the energy estimates.
On the other hand, the conformal energy estimates allow one to get the $L^2$--type estimates for wave components (i.e. the cases of $m = 0$), but they cannot be applied any more due to the presence of the mass term, which means we cannot obtain $L^2$--type estimates for $v_i$ using the conformal energies.
Second, the solution to the system \eqref{eq:model-mKG} does not decay sufficiently fast. In general we can expect solutions to Klein-Gordon equations with fixed mass to decay like $t^{-3/2}$ in $\RR^{3+1}$, but due the possibly vanishing mass $m \in [0, 1]$, the best we can expect for the Klein-Gordon components $v_i$ is the (mass independent) wave decay, i.e.
$$
|v_i| \lesssim t^{-1},
$$
and the (mass dependent) Klein-Gordon decay
$$
|v_i| \lesssim m^{-1} t^{-3/2}.
$$
Besides, the null form of the type $\del^\alpha v_j \del_\alpha v_k$ does not seem to decay sufficiently fast. Because it is not consistent with the Klein-Gordon equations, since we need to use the scaling vector field to gain a good factor of $t^{-1}$ from $\del^\alpha v_j \del_\alpha v_k$, but the scaling vector field is not consistent with the Klein-Gordon equations. Last but not least, there are some difficulties in gaining the factor $t^{-1}$ from the null forms $Q_0 (v_j, v_k), Q_{\alpha \beta}(v_j, v_k)$ in the highest order energy, which is due to the lack of the conformal energy estimate again.

In order to tackle the problems brought by the presence but possibly vanishing mass term, we will rely on the following observations and insights. First, we will use the hyperboloidal foliation of the spacetime to prove the (uniform) global existence result for the system \eqref{eq:model-mKG}, which is developed by Klainerman \cite{Klainerman85}, Hormander \cite{Hormander}, LeFloch-Ma \cite{PLF-YM-book}, and Klainerman-Wang-Yang \cite{Klainerman-QW-SY}, etc. We will take the advantage that the null forms ($Q_0, Q_{\alpha \beta}$) can be decomposed as sums of products of good components in the hyperboloidal setting (see Lemma \ref{lem:est-null-h}), and this is true even for the highest order energy. As a consequence, we can obtain the mass dependent pointwise decay result
$$
\big| v_i (t, x) \big| \lesssim m^{-1} t^{-3/2}.
$$
Next, we will move to the usual flat foliation of the spacetime to show the unified pointwise decay result. To achieve this, we will obtain the mass independent $L^2$--type estimates for the solution $v = (v_i)$ by using tricks from the Fourier analysis. To be more precise, we write the Klein-Gordon equation in the Fourier space, and solve the corresponding ordinary differential equation to get the solution in the Fourier space, and then obtain the mass independent $L^2$--type estimates for the solution $v = (v_i)$ (see Proposition \ref{prop:Va-key}). However, according to Proposition \ref{prop:Va-key}, we need to gain the factor $t^{-1}$ from the null nonlinearities to get sufficiently good $L^2$--type estimates for the solution. For the null forms of type $Q_{\alpha \beta}$, we easily have
$$
\big| Q_{\alpha \beta} (v_j, v_k) \big|
\lesssim {1\over t} \big( |\Gamma v_j| |\del v_k| + |\Gamma v_k| |\del v_j|  \big),
\qquad
\Gamma \in \{ \del_\alpha, \Omega_{ab}, L_a \},
$$
and for the high order case, the observation (from \cite{Katayama12a}) helps
$$
Q_{\alpha \beta} (v_j, v_k) = \del_\alpha (v_j \del_\beta v_k) - \del_\beta (\del_\alpha v_k v_j),
$$
which is thanks to the hidden divergence form of the null nonlinearities $Q_{\alpha \beta}$.
At a first glance, it does not seem to be possible to gain the factor $t^{-1}$ from the null form $Q_0$, which is because we do not have any good control on the scaling vector field, but we observe that a nonlinear transformation will transform the quadratic term $Q_0$ to the sum of cubic terms and quadratic terms with a good factor $m^2$ in front. These observations allow us to obtain the mass independent $L^2$--type estimates and hence the pointwise decay result for the solution $v = (v_i)$. More details following, we will divide the solution into several parts, and conduct the analysis on each part according to their features. For the parts where we can gain $t^{-1}$ factor from the null form or the nonlinearities are cubic, Proposition \ref{prop:Va-key} will be sufficient to obtain the mass independent pointwise decay. For the part with quadratic nonlinearities with the factor $m^2$, we will be carefully study the $m$ dependent relation of the norms of the nonlinearities, and try to gain the factor $m$ to cancel the one appearing in the energy.

\subsection{Main theorem}

Now we provide the statement of the main result.

\begin{theorem}\label{thm:main-mKG}
Consider the systems of Klein-Gordon equations \eqref{eq:model-mKG} with mass $m \in [0, 1]$, and let $N \geq 6$ be an integer. There exists small $\eps_0 > 0$, such that for all $\eps < \eps_0$, and all compactly supported initial data which are small in the sense that
\be 
\| v_{i0} \|_{H^{N+1}} + \| v_{i1} \|_{H^N}
\leq \eps, 
\qquad \text{for all i,}
\ee
then the Cauchy problem \eqref{eq:model-mKG}--\eqref{eq:mKG-ID} admits a global-in-time solution $v = (v_i)$. In addition, the solution decays uniformly in terms of the mass parameter $m$ as
\bel{eq:m-decay}
\big| v_i(t, x) \big|
\lesssim
{1 \over t + m t^{3/2}}.
\ee
\end{theorem}

In the proof of Theorem \ref{thm:main-mKG}, we will always assume $m \in (0, 1]$ unless specified since the result for the case of $m = 0$ is classical. We will also assume the initial data $(v_{i0}, v_{i1})$ are spatially supported in the unit ball $\{(x, t) : t = t_0 = 2, |x| \leq 1 \}$, but the results in the theorem still holds for all of the initial data with compact support, see the remark in \cite{DW}. We note that the compactness assumption implies
$$
\| v_{i1} \|_{L^{6/5}(\RR^3)}
\lesssim \| v_{i1} \|_{L^2(\RR^3)},
$$
and this will be used when applying Proposition \ref{prop:Va-key}.

It can be seen from Theorem \ref{thm:main-mKG} that the global existence result and the pointwise decay result are both consistent with the cases of $m = 0$ and $m = 1$, which are the usual wave equations and the usual Klein-Gordon equations (with fixed mass). 
Worth to mention, the unified decay result \eqref{eq:m-decay} shows that the solution decays more like a wave component (with no $m$ dependence) as $t^{-1}$ in the time range $t \in [t_0, m^{-2})$, while it decays more like a Klein-Gordon component (with $m$ dependence) as $m^{-1} t^{-3/2}$ in the rest part of the time range (if non-empty).
In addition to the results contained in Theorem \ref{thm:main-mKG}, we have the following convergence result, which tells us that the solution to the system \eqref{eq:model-mKG} converges to the solution to the corresponding wave system (i.e. the system \eqref{eq:model-mKG} with $m=0$) in certain sense. Let
$$
v^{(m)}, 
\qquad
m \in [0, 1],
$$
denote the solution to the system \eqref{eq:model-mKG} with mass $m$, and we can now demonstrate the convergence theorem.

\begin{theorem}\label{thm:converge}
Consider the system \eqref{eq:model-mKG}, and let the same assumptions in Theorem \ref{thm:main-mKG} hold. Then the solution to the system \eqref{eq:model-mKG} with mass $m$ converges to the system \eqref{eq:model-mKG} with $m = 0$, in the sense that (with $0 < \delta \ll 1$)
\bel{eq:thm-converge} 
\big\| \del \del^I L^J ( v^{(m)} - v^{(0)} ) \big\|_{L^2}
+
m \big\| \del^I L^J ( v^{(m)} - v^{(0)} ) \big\|_{L^2}
\lesssim m^2 t^{1+\delta},
\qquad
|I| + |J| \leq N.
\ee
\end{theorem}

We note that Theorem \ref{thm:converge} indicates that the solution $v^{(m)}$ tends to $v^{(0)}$ at the rate $m^2$ when $m \to 0$ on each fixed slice $t = constant$, but the bounds for the energy of the difference $v^{(m)} - v^{(0)}$ blow up as $t \to +\infty$ for each fixed $m$.

\subsection{Outline}

The rest of this paper is organised as follows: In Section \ref{sec:pre}, we revisit some notations, Sobolev--type inequalities, and basic results on the Klein-Gordon equations. Next, we provide the key result on obtaining mass independent $L^2$ norm estimates for solutions to the Klein-Gordon equations with possibly vanishing masses in Section \ref{sec:key}. Then we prove the global existence result in Section \ref{sec:existence}. Finally, the proof for the mass independent decay result and the proof for the convergence result are illustrated in Section \ref{sec:decay} and Section \ref{sec:converge} respectively.


\section{Preliminaries}\label{sec:pre}

\subsection{Basic notations}

We work in the $(3+1)$ dimensional spacetime with metric $\eta = \text{diag} (-1, 1, 1, 1)$. We write a point $(x^0, x^a) = (t, x^a)$, and the indices are raised or lowered by the metric $\eta$. We use 
$$
\aligned
&\del_\alpha = \del_{x^\alpha}, 
&\qquad \qquad \alpha = 0, 1, 2, 3,
\\
&\Omega_{ab} = x_a \del_b - x_b \del_a, 
&\qquad \qquad a, b = 1, 2, 3, \text{ and } a<b,
\\
&L_a = x_a \del_t + t \del_a,
&\qquad \qquad a = 1, 2, 3,
\endaligned
$$ 
to denote the vector fields of translation, rotation, and Lorentz boosts respectively. For convenience, we use $\del, \Omega, L$ to represent a general vector field of translation, rotation, and Lorentz boost respectively, and with the notation 
$$
V = \{ \del_\alpha, \Omega_{ab}, L_a \},
$$
$\Gamma$ is used to represent a general vector field in $V$.

When it turns to the hyperboloidal foliation of the spacetime of the cone $\Kcal := \{(t, x) : t \geq t_0 = 2, t \geq |x| + 1 \}$, we use $\Hcal_s = \{(t, x) : t^2 = |x|^2 + s^2 \}$ to denote a hyperboloid at hyperbolic time $s$ with $s \geq s_0 = 2$. We note that throughout we will only consider (unless specified) functions with support in $\Kcal$, since the solution to \eqref{eq:model-mKG} is supported in $\Kcal$. We emphasize here that for all points $(t, x) \in \Kcal \bigcap \Hcal_s$ ($s \geq 2$), the following relations hold
\be 
s \leq t \leq s^2,
\qquad
|x| \leq t.
\ee
In order to adapt to the hyperboloidal foliation of the spacetime, we first recall the semi-hyperboloidal frame introduced in \cite{PLF-YM-book}, which is defined by
\be 
\underdel_0 = \del_t,
\qquad
\underdel_a = {L_a \over t} = \del_a + {x_a \over t} \del_t.
\ee
We can also represent the usual partial derivatives $\del_\alpha$ in terms of the semi-hyperboloidal frame by
\be 
\del_0 = \underdel_0,
\qquad
\del_a = - {x_a \over t} \underdel_0 + \underdel_a.
\ee

We denote the Fourier transform of a nice function $u$ by
$$
\uh (\xi) = \int_{\RR^3} u(x) e^{-2\pi x \cdot \xi} \, dx.
$$
We recall some properties regarding the Fourier transform, which will be used in the analysis.
The partial derivatives are reflected by Fourier multipliers
\be 
\widehat{\del_a u} (\xi)
= 2\pi i \xi_a \uh(\xi),
\ee
and the Plancheral identity connects the $L^2$ norms between the function and its Fourier transform
\be
\| u \|_{L^2(\RR^3)}
= \big\| \uh \big\|_{L^2(\RR^3)}.
\ee

\subsection{Estimates for commutators and null forms}

\paragraph{Estimates for commutators}

We first demonstrate some well-known results regarding the commutators of different vector fields, which can be found in \cite{Sogge, PLF-YM-book}.

\begin{lemma}\label{lem:comm-est11}
Let $u$ be a sufficiently regular function with support $\Kcal$ and denote the commutator by $[\Gamma, \Gamma'] = \Gamma \Gamma'- \Gamma'\Gamma$, then we have
\be 
\aligned
\big|[\del_\alpha, L_a] u \big| + \big| [\del_\alpha, \Omega_{ab}] u \big|
&\lesssim
\sum_\beta \big| \del_\beta u \big|,
\\
\big| [L_c, \Omega_{ab}] u \big| + \big| [L_a, L_b] u \big|
&\lesssim
\sum_{d} \big| L_d u \big|,
\\
\big| [L_a, (s/t)] u \big|
&\lesssim
\big| (s/t) u \big|,
\\
\big| [L_b L_a, (s/t)] u \big| 
&\lesssim
\big| (s/t) u \big| + \sum_c \big| (s/t) L_c u \big|,
\\
\big| [\underdel_a, L_b] u \big|
&\lesssim
\sum_c \big| \underdel_c u \big|.
\endaligned
\ee
\end{lemma}

Next, we recall the following result from \cite{Sogge}, which tells us that the null forms acted by a vector field still give us null forms.

\begin{lemma}\label{lem:commu-null}
For all nice functions $u, w$ we have
\be
\aligned
\del_\alpha Q_0 (u, w)
- Q_0 (\del_\alpha u, w) - Q_0 (u, \del_\alpha w)
& = 0,
\\
\del_\gamma Q_{\alpha \beta} (u, w)
- Q_{\alpha \beta} (\del_\gamma u, w) - Q_{\alpha \beta} (u, \del_\gamma w)
&= 0,
\\
L_a Q_0 (u, w)
- Q_0 (L_a u, w) - Q_0 (u, L_a w)
&= 0,
\\
\big| L_a Q_{\alpha \beta} (u, w) 
- Q_{\alpha \beta} (L_a u, w) - Q_{\alpha \beta} (u, L_a w) \big|
&\leq \sum_{\alpha', \beta'} \big| Q_{\alpha' \beta'} (u, w) \big|.
\endaligned
\ee
\end{lemma}

\paragraph{Estimates for null forms}

We first recall the classical estimates for null forms of the type $Q_{\alpha \beta}$, which can be found in \cite{Sogge}.

\begin{lemma}\label{lem:est-null-f}
We have for sufficiently regular functions $u, w$ with support in $\Kcal = \{ (t, x) : t \geq t_0, t \geq |x| + 1 \}$
\be 
\big| Q_{\alpha \beta} (u, w) \big|
\lesssim
{1\over t} \big( |L u | |\del w| + | \del u| |L w|  \big),
\ee

\end{lemma}

Besides, following from \cite{PLF-YM-cmp} of the hyperboloidal setting, we also have the following estimates for all types of null forms.

\begin{lemma}\label{lem:est-null-h}
It holds for smooth functions $u, w$ with support in $\Kcal = \{ (t, x) : t \geq t_0, t \geq |x| + 1 \}$ that
\be
\big| Q_0 (u, w) \big|
+ \big| Q_{\alpha \beta} (u, w) \big|
\lesssim
(s/t)^2 |\del_t u \del_t w| + \sum_{a, \alpha} \big( |\underdel_\alpha u \underdel_a w | + |\underdel_\alpha w \underdel_a u | \big).
\ee
\end{lemma}

\begin{proof}
We revisit the proof for $Q_0 (u, w)$ only, from \cite{PLF-YM-book}, for readers who are not familiar with the hyperboloidal foliation method.

Recall the semi-hyperboloidal frame 
$$
\del_t = \underdel_0,
\qquad
\del_a = -{x_a \over t} \underdel_0 + \underdel_a,
$$
and we express the null form $Q_0 (u, w)$ in the semi-hyperboloidal frame to get
$$
Q_0 (u, w)
=
- {s^2 \over t^2} \underdel_0 u \underdel_0 w - {x_a \over t} (\underdel_0 u \underdel^a w + \underdel_0 w \underdel^a u ) + \underdel_a u \underdel^a w.
$$
Then the fact $|x| \leq t$ concludes the estimates.

\end{proof}

\subsection{Sobolev--type inequalities}

\paragraph{Klainerman-Sobolev inequality}

In order to obtain pointwise decay estimates for the Klein-Gordon components, we need the following Klainerman-Sobolev inequality, which was introduced in \cite{Klainerman2}. The reason why we need the following version of Klainerman-Sobolev inequality is that one will not need to rely on the scaling vector field $L_0 = t \del_t + x^a \del_a$ (which is not consistent with the Klein-Gordon equations), and this feature is vital in obtaining the mass independent pointwise decay results for the Klein-Gordon components.

\begin{proposition}\label{prop:K-S}
Assume $u = u(t, x)$ is a sufficiently smooth function which decays sufficiently fast at space infinity for each fixed $t \geq 2$,
then for any $t \geq 2$, $x \in \RR^3$, we have
\bel{eq:Va-K-S}
|u(t, x)|
\lesssim t^{-1} \sup_{0\leq t' \leq 2t, |I| \leq 3} \big\| \Gamma^I u \big\|_{L^2(\RR^3)},
\qquad
\Gamma \in V = \{ L_a, \del_\alpha, \Omega_{ab} = x^a \del_b - x^b \del_a \}.
\ee
\end{proposition}

We will use a simplified version of Proposition \ref{prop:K-S}, where we do not need to use the rotation vector field because we only need to consider function supported in $\Kcal = \{(t, x) : t \geq 2, t \geq |x| + 1 \}$.

\begin{proposition}\label{prop:K-S2}
Assume $u = u(t, x)$ is a sufficiently smooth function with support $\Kcal$,
then for any $t \geq 2$, $x \in \RR^3$, we have
\bel{eq:Va-K-S2}
|u(t, x)|
\lesssim t^{-1} \sup_{t_0\leq t' \leq t_0 + 2t, |I| + |J| \leq 3} \big\| \del^I L^J u \big\|_{L^2(\RR^3)}.
\ee
\end{proposition}
\begin{proof}
The Klainerman-Sobolev inequality \eqref{eq:Va-K-S2} can be obtained from \eqref{eq:Va-K-S}, the commutator estimates, and the fact that 
$$
\sum_{ab} | \Omega_{ab} w |
\lesssim \sum_a | L_a w |
$$
holds for all nice functions $w$ with support $\Kcal$.

\end{proof}

\paragraph{Sobolev-type inequality on hyperboloids}

We now recall a Sobolev-type inequality adapted to the hyperboloids from \cite{PLF-YM-book}, which allows us to get the (mass dependent) sup-norm estimates for the Klein-Gordon components.

\begin{proposition} \label{prop:sobolev}
Let $u= u(t, x)$ be a sufficiently nice function with support $\{(t, x): t \geq |x| + 1\}$, then for all  $s \geq 2$, one has 
\bel{eq:Sobolev2}
\sup_{\Hcal_s} \big| t^{3/2} u(t, x) \big|  
\lesssim \sum_{| J |\leq 2} \| L^J u \|_{L^2_f(\Hcal_s)},
\ee
where the symbol $L$ denotes the Lorentz boosts. 
\end{proposition}

The Sobolev inequality \eqref{eq:Sobolev2} combined with the commutator estimates gives us the following inequality
\be
\sup_{\Hcal_s} \big| s \hskip0.03cm t^{1/2} u(t, x) \big|  
\lesssim \sum_{| J |\leq 2} \| (s/t) L^J u \|_{L^2_f(\Hcal_s)}.
\ee

\paragraph{Hardy inequality on hyperboloids}

\begin{proposition}
Let $u = u(x)$ be a sufficiently smooth function in dimension $d \geq 3$, then it holds ($r = |x|$)
\be 
\big\| r^{-1} u \big\|_{L^2(\RR^d)}
\leq 
C \sum_a \big\| \del_a \big\|_{L^2(\RR^d)}.
\ee
\end{proposition}

The Hardy inequality can also be adapted to the hyperboloidal setting, see for instance \cite{PLF-YM-book, PLF-YM-cmp}.

\begin{proposition}
Assume the function $u$ is sufficiently regular and supported in the region $\Kcal$, then for all $s\geq 2$, one has 
\bel{eq:Hardy} 
\| r^{-1} u \|_{L^2_f(\Hcal_s)}
\lesssim
\sum_a \| \underdel_a u \|_{L^2_f(\Hcal_s)}.
\ee
As a consequence, we also have
\bel{eq:Hardy2}
\| t^{-1} u \|_{L^2_f(\Hcal_s)}
\lesssim
\sum_a \| \underdel_a u \|_{L^2_f(\Hcal_s)}.
\ee
\end{proposition}

\paragraph{Sobolev embedding theorem}

We recall the following type of Sobolev embedding theorem.

\begin{proposition}
Let $u = u(x) \in L^{6/5}(\RR^3)$, then it holds that
\bel{eq:S-embedding}
\Big\| {u \over \Lambda}  \Big\|_{L^2(\RR^3)} 
\lesssim \| u \|_{L^{6/5}(\RR^3)},
\ee
in which $\Lambda = \sqrt{- \Delta} = \sqrt{- \del_a \del^a}$.
\end{proposition}

\subsection{Energy estimates for Klein-Gordon equations}

Given a function $u = u(t, x)$ supported in $\Kcal$, we define its energy $\Ecal_m$, following \cite{PLF-YM-book}, on a hyperboloid $\Hcal_s$ by
\bel{eq:2energy} 
\aligned
\Ecal_m(s, u)
&:=
 \int_{\Hcal_s} \Big( \big(\del_t u \big)^2+ \sum_a \big(\del_a u \big)^2+ 2 (x^a/t) \del_t u \del_a u + m^2 u^2 \Big) \, dx
\\
&= \int_{\Hcal_s} \Big( \big( (s/t)\del_t u \big)^2+ \sum_a \big(\underdel_a u \big)^2+ m^2 u^2 \Big) \, dx
\\
&= \int_{\Hcal_s} \Big( \big( \underdel_\perp u \big)^2+ \sum_a \big( (s/t)\del_a u \big)^2+ \sum_{a<b} \big( t^{-1}\Omega_{ab} u \big)^2+ m^2 u^2 \Big) \, dx,
\endaligned
\ee
in which $\underdel_{\perp} := \del_t+ (x^a / t) \del_a$ is the orthogonal vector field. 
The integral $L^2_f(\Hcal_s)$ is defined by
\bel{flat-int}
\| u \|_{L^2_f(\Hcal_s)}^2
:=\int_{\Hcal_s}| u |^2 \, dx 
:=\int_{\RR^3} \big| u(\sqrt{s^2+|x|^2}, x) \big|^2 \, dx.
\ee
We note that it holds
$$
\big\| (s/t) \del u \big\|_{L^2_f(\Hcal_s)} + \sum_a \big\| \underdel_a u \big\|_{L^2_f(\Hcal_s)}
\lesssim
\Ecal_m(s, u)^{1/2},
$$
which will be frequently used.

Next, we demonstrate the energy estimates to the hyperboloidal setting.

\begin{proposition}[Energy estimates for wave-Klein-Gordon equations]
For $m \geq 0$ and for $s \geq s_0$ (with $s_0 = 2$), it holds that
\bel{eq:w-EE} 
\Ecal_m(s, u)^{1/2}
\leq 
\Ecal_m(s_0, u)^{1/2}
+ \int_2^s \big\| -\Box u + m^2 u \big\|_{L^2_f(\Hcal_{s'})} \, ds'
\ee
for all sufficiently regular functions $u$, which are defined and supported in $\Kcal_{[s_0, s]} = \bigcup_{s_0 \leq s'\leq s} \Hcal_{s'}$.
\end{proposition}
The proof of \eqref{eq:w-EE} can be found in \cite{PLF-YM-book, PLF-YM-cmp}.

In comparison with $\Ecal_m$, we use $E_m$ to denote the usual energy on the flat slices $t = constant$, which is expressed as
$$
E_m (t, u)
=
\int_{\RR^3} \big| \del u \big|^2 + m^2 u^2 \, dx.
$$
Similarly, we have the following energy estimate
\bel{eq:w-EE2} 
E_m(t, u)^{1/2}
\leq 
E_m(t_0, u)^{1/2}
+ \int_2^t \big\| -\Box u + m^2 u \big\|_{L^2(\RR^3)} \, dt'.
\ee


\section{Mass independent $L^2$ norm estimates for Klein-Gordon equations}\label{sec:key}

We will rely on the following key proposition to obtain the mass independent $L^2$--type energy estimates for the solution to the system \eqref{eq:model-mKG}. A similar result was obtained in \cite{Dong1905}, and we now provide an enhanced version of it.

\begin{proposition}\label{prop:Va-key}
Consider the wave-Klein-Gordon equation
$$
-\Box u + m^2 u 
= f,
\qquad
\big( u, \del_t u \big)(t_0)
= (u_0, u_1),
$$
with mass $m \in [0, 1]$, and assume 
$$
\| u_0 \|_{L^2(\RR^3)}
+
\| u_1 \|_{L^2(\RR^3) \bigcap L^{6/5}(\RR^3)}
\lesssim C_{t_0},
\quad
\| f \|_{L^{6/5}(\RR^3)}
\leq C_f t^{-1 + q},
$$
for some numbers $C_{t_0}$ and $C_f$.
Then we have
\begin{eqnarray}\label{eq:keyLemma}
\| u \|_{L^2(\RR^3)}
\lesssim 
\left\{
\begin{array}{lll}
C_{t_0} + C_f t^q, & \quad q>0,
\\
C_{t_0} + C_f \log t, & \quad q = 0,
\\
C_{t_0} + C_f, & \quad q<0.
\end{array}
\right.
\end{eqnarray}
\end{proposition}
\begin{proof}

We first write the equation of $u$ in the Fourier space $(t, \xi)$
$$
\del_t \del_t \uh + \xi_m^2 \uh
= \widehat{f},
$$
and solve the ordinary differential equation to get the solution
$$
\uh(t, \xi)
= \cos\big( t \xi_m \big) \uh_0
+ { \sin\big( t \xi_m \big) \over \xi_m } \uh_1
+ {1\over \xi_m} \int_{t_0}^{t} \sin\big( (t-t') \xi_m \big) \widehat{f}(t') \, dt',
$$
with the notations defined by
$$
\uh_0 = \widehat{u_0},
\qquad
\uh_1 = \widehat{u_1},
\qquad
\xi_m = \sqrt{4 \pi^2 |\xi|^2 + m^2} \geq |\xi|.
$$

Next, we take $L^2$ norm in the frequency space to obtain
$$
\aligned
&\big\| \uh(t, \cdot) \big\|_{L^2(\RR^3)}
\\
\lesssim 
&\big\| \cos\big( t \xi_m \big) \uh_0 \big\|_{L^2(\RR^3)}
+\Big\|  { \sin\big( t \xi_m \big) \over \xi_m } \uh_1 \Big\|_{L^2(\RR^3)}
+ \Big\| {1\over \xi_m} \int_{t_0}^{t} \sin\big( (t-t') \xi_m \big) \widehat{f}(t') \, dt' \Big\|_{L^2(\RR^3)}
\\
\lesssim 
& \big\| \uh_0 \big\|_{L^2(\RR^3)}
+\Big\|  { 1 \over |\xi| } \uh_1 \Big\|_{L^2(\RR^3)}
+   \int_{t_0}^{t} \Big\| {1\over |\xi|} \widehat{f}(t') \Big\|_{L^2(\RR^3)} \, dt', 
\endaligned
$$
which in the physical space reads
$$
\aligned
\big\| u(t, \cdot) \big\|_{L^2(\RR^3)}
\lesssim 
\big\| u_0 \big\|_{L^2(\RR^3)}
+\Big\|  { u_1 \over \Lambda } \Big\|_{L^2(\RR^3)}
+   \int_{t_0}^{t} \Big\| {f (t')\over \Lambda} \Big\|_{L^2(\RR^3)} \, dt',
\endaligned
$$
with $\Lambda = \sqrt{-\del_a \del^a}$.

Then by the Sobolev embedding theorem \eqref{eq:S-embedding}, we admit
$$
\big\| u(t, \cdot) \big\|_{L^2(\RR^3)}
\lesssim 
\big\| u_0 \big\|_{L^2(\RR^3)}
+\big\| u_1 \big\|_{L^{6/5}(\RR^3)}
+  \int_{t_0}^{t} \big\| f (t') \big\|_{L^{6/5}(\RR^3)} \, dt', 
$$
and the simple result of the integral 
$$
\int_{t_0}^t t'^{-1 + q} \, dt'
$$
implies the desired result \eqref{eq:keyLemma}.
\end{proof}

We note in the proof that in order to obtain the $L^2$--type energy estimates for the solution we transformed the original equation to the Fourier space and solve the corresponding ordinary differential equation and then conduct the analysis. Worth to mention, we find that such procedures to bound the $L^2$--type norms for Klein-Gordon equations (with possibly vanishing mass) can also be applied to pure wave equations, especially in the low dimension where the conformal energy cannot bound the $L^2$ norm of the solution, see for instance \cite{Dong1910}.

As a consequence, combined with the Klainerman-Sobolev inequality \eqref{eq:Va-K-S} we also get
\begin{eqnarray}\label{eq:keyLemma2}
\| u \|_{L^\infty(\RR^3)}
\lesssim 
\left\{
\begin{array}{lll}
C_{t_0} t^{-1} + C_f t^{-1+q}, & \quad q>0,
\\
C_{t_0} t^{-1} + C_f t^{-1} \log t, & \quad q = 0,
\\
C_{t_0} t^{-1} + C_f t^{-1}, & \quad q<0,
\end{array}
\right.
\end{eqnarray}
if additional information for higher order energies are true (with $\Gamma \in \{\del, \Omega, L \}$)
$$
\big\| \Gamma^I u(t_0) \big\|_{L^2(\RR^3)}
+
\big\| \Gamma^I \del_t u(t_0) \big\|_{L^2(\RR^3) \bigcap L^{6/5}(\RR^3)}
\lesssim C_{t_0},
\quad
\big\| \Gamma^I f \big\|_{L^{6/5}(\RR^3)}
\leq C_f t^{-1 + q},
\qquad
|I|  \leq 3.
$$


\section{Proof for the global existence result}\label{sec:existence}

\subsection{Initialisation of the bootstrap method}

In this section, we aim to prove the uniform global existence result for the system \eqref{eq:model-mKG} via the hyperboloidal foliation method. 

As usual, we will rely on the bootstrap method. The local well-posedness result allows us to assume (for all $i$) 
\bel{eq:BA-mKG}
\aligned
\Ecal_m (s, \del^I L^J v_i)^{1/2}
\leq C_1 \eps,
\qquad
|I| + |J| \leq N,
\endaligned
\ee
for all $s \in [s_0, s_1)$ with $s_1 > s_0$. In \eqref{eq:BA-mKG}, $C_1 \ll 1$ is some large constant to be determined, and $\eps > 0$ is the size of the initial data satisfying $C_1 \eps \ll \delta \ll 1$, and $s_1$ is defined by
\bel{eq:def-s1}
s_1 := \sup \{ s : s > s_0, \eqref{eq:BA-mKG} ~holds\}.
\ee
If $s_1 = + \infty$, then the global existence result is done. So in the following proof, we first assume $s_1 < + \infty$ and then deduce contradictions to assert that $s_1$ must be $+\infty$.

By recalling the definition of the energy $\Ecal_m$, we easily have the following estimates.

\begin{lemma}
Assume \eqref{eq:BA-mKG} holds, then for all $s \in [s_0, s_1)$ and $|I| + |J| \leq N$ we have the following estimates 
\be 
\aligned
\big\| (s/t) \del \del^I L^J v_i \big\|_{L^2_f(\Hcal_s)} + m \big\| \del^I L^J v_i \big\|_{L^2_f(\Hcal_s)} + \sum_a \big\| \underdel_a \del^I L^J v_i \big\|_{L^2_f(\Hcal_s)}
&\lesssim C_1 \epsilon, 
\\
\big\| (s/t) \del^I L^J \del v_i \big\|_{L^2_f(\Hcal_s)} + \sum_a \big\| \del^I L^J \underdel_a v_i \big\|_{L^2_f(\Hcal_s)}
&\lesssim C_1 \epsilon.
\endaligned
\ee
\end{lemma}

Combined with the Sobolev--type inequality on hyperboloids \eqref{eq:Sobolev2}, the following pointwise estimates are valid.

\begin{lemma}\label{lem:BA1-consqt}
For all $|I| + |J| \leq N-2$ we have
\be 
\big| (s/t) \del \del^I L^J v_i \big| + m \big| \del^I L^J v_i \big| + \sum_a \big| \underdel_a \del^I L^J v_i \big|
\lesssim C_1 \eps t^{-3/2}.
\ee

\end{lemma}

Besides of the estimates above, we also introduce estimates obtained by using the Hardy inequality \eqref{eq:Hardy}--\eqref{eq:Hardy2}. They will not be used in the current section, but will be used in Section \ref{sec:decay}.

\begin{lemma}\label{lem:h-est0}
The following estimates are valid
\be 
\aligned
\big\| t^{-1} \del^I L^J v_i \big\|_{L^2_f(\Hcal_s)}
\lesssim C_1 \eps,
\qquad
|I| + |J| \leq N,
\\
\big| \del^I L^J v_i \big|
\lesssim C_1 \eps t^{-1/2},
\qquad
|I| + |J| \leq N-2.
\endaligned
\ee
\end{lemma}

\subsection{Improved energy estimates and global existence result}

We now want to show the improved energy estimates for the solution $v = (v_i)$, and then conclude the global existence result.

\begin{proposition}[Improved energy estimates]\label{prop:Imp-EE}
Let the assumptions in\eqref{eq:BA-mKG} be true, then for all $s \in [s_0, s_1)$ it holds that
\bel{eq:Imp-EE}
\Ecal_m (s, \del^I L^J v_i)^{1/2}
\lesssim \eps + (C_1 \eps)^2,
\qquad
|I| + |J| \leq N.
\ee
\end{proposition}
\begin{proof}
Acting the vector field $\del^I L^J$ with $|I| + |J| \leq N$ on the model equations in \eqref{eq:model-mKG}, we get
$$
- \Box \del^I L^J v_i + m^2 \del^I L^J v_i 
= N^{jk}_i \del^I L^J Q_0 (v_j, v_k) + M^{jk \alpha \beta}_i \del^I L^J Q_{\alpha \beta} (v_j, v_k).
$$
According to the commutator estimates in Lemma \ref{lem:commu-null}, we can bound the right hand side as follows
$$
\aligned
&\big| N^{jk}_i \del^I L^J Q_0 (v_j, v_k) + M^{jk \alpha \beta}_i \del^I L^J Q_{\alpha \beta} (v_j, v_k) \big|
\\
\lesssim
&\sum_{\substack{j, k, \alpha, \beta \\ |I_1| + |I_2| \leq |I|, |J_1| + |J_2| \leq |J| }} \Big( \big| Q_0 (\del^{I_1}L^{J_1} v_j, \del^{I_2}L^{J_2} v_k) \big| + \big| Q_{\alpha \beta} (\del^{I_1}L^{J_1} v_j, \del^{I_2}L^{J_2} v_k) \big| \Big).
\endaligned
$$
Then the estimates for null forms in the hyperboloidal setting in Lemma \ref{lem:est-null-h} yield
$$
\aligned
&\big| N^{jk}_i \del^I L^J Q_0 (v_j, v_k) + M^{jk \alpha \beta}_i \del^I L^J Q_{\alpha \beta} (v_j, v_k) \big|
\\
\lesssim
&\sum_{\substack{j, k, \alpha, a \\ |I_1| + |I_2| \leq |I|, |J_1| + |J_2| \leq |J| }} \Big( (s/t)^2 \big| \del_t \del^{I_1}L^{J_1} v_j \del_t \del^{I_2}L^{J_2} v_k \big| 
\\
& \hskip3.4cm + \big|\underdel_\alpha \del^{I_1}L^{J_1} v_j \underdel_a \del^{I_2}L^{J_2} v_k \big| + \big|\underdel_\alpha \del^{I_2}L^{J_2} v_k \underdel_a \del^{I_1}L^{J_1} v_j \big|\Big).
\endaligned
$$

We rely on the energy estimates for Klein-Gordon equations on hyperboloids \eqref{eq:w-EE} to get
$$
\aligned
&\Ecal_m (s, \del^I L^J v_i)^{1/2}
\\
\leq 
&\Ecal_m (s_0, \del^I L^J v_i)^{1/2}
+
\int_{s_0}^s \big\| N^{jk}_i \del^I L^J Q_0 (v_j, v_k) + M^{jk \alpha \beta}_i \del^I L^J Q_{\alpha \beta} (v_j, v_k) \big\|_{L^2_f(\Hcal_{s'})} \, ds'
\\
\lesssim
&\epsilon
+  \sum_{\substack{j, k, \alpha, a \\ |I_1| + |I_2| \leq |I|, |J_1| + |J_2| \leq |J| }} 
\int_{s_0}^s \Big\| (s/t)^2 \big| \del_t \del^{I_1}L^{J_1} v_j \del_t \del^{I_2}L^{J_2} v_k \big| 
\\
& \hskip3.4cm + \big|\underdel_\alpha \del^{I_1}L^{J_1} v_j \underdel_a \del^{I_2}L^{J_2} v_k \big| + \big|\underdel_\alpha \del^{I_2}L^{J_2} v_k \underdel_a \del^{I_1}L^{J_1} v_j \big|\Big\|_{L^2_f(\Hcal_{s'})} \, ds'.
\endaligned
$$
We thus have (since\footnote{Actually $N\geq 4$ is enough to ensure the global existence result.} $N \geq 6$)
$$
\aligned
&\Ecal_m (s, \del^I L^J v_i)^{1/2}
\\
\lesssim
&\epsilon
+ \sum_{\substack{j, k, \alpha, a \\ |I_1| + |J_1| \leq N, |I_2| + |J_2| \leq N-2}} 
\int_{s_0}^s \Big( \big\| (s'/t) \del_t \del^{I_1}L^{J_1} v_j \big\|_{L^2_f(\Hcal_{s'})} \big\| (s'/t) \del_t \del^{I_2}L^{J_2} v_k \big\|_{L^\infty_f(\Hcal_{s'})} 
\\
&\hskip3.5cm + \big\| (s'/t) \underdel_\alpha \del^{I_1}L^{J_1} v_j  \big\|_{L^2_f(\Hcal_{s'})} \big\| (t/s') \underdel_a \del^{I_2}L^{J_2} v_k \big\|_{L^\infty_f(\Hcal_{s'})}
\\
&\hskip3.5cm + \big\| \underdel_a \del^{I_1}L^{J_1} v_j \big\|_{L^2_f(\Hcal_{s'})} \big\| \underdel_\alpha \del^{I_2}L^{J_2} v_k  \big\|_{L^\infty_f(\Hcal_{s'})}
\Big) \, ds'
\\
\lesssim
& \eps + (C_1 \eps)^2 \int_{s_0}^s \big( t^{-3/2} + t^{-1/2} s'^{-1}  \big) \, ds'
\lesssim \eps + (C_1 \eps)^2.
\endaligned
$$
Hence the proof is done.

\end{proof}

As a consequence of the improved energy estimates in Proposition \ref{prop:Imp-EE}, we conclude the global existence result of the system \eqref{eq:model-mKG}.

\begin{proof}[Proof of the global existence result]\label{proof:existence}

We choose $C_1$ large and $\eps$ small such that 
$$
C \big( \eps + (C_1 \eps)^2 \big)
\leq {1\over 2} C_1 \eps,
$$
in which $C$ is the hidden constant in \eqref{eq:Imp-EE}, which thus leads us to the improved estimates
$$
\Ecal_m (s, \del^I L^J v_i)^{1/2}
\leq {1\over 2} C_1 \eps,
\qquad
|I| + |J| \leq N.
$$
If $s_1 > s_0$ is some finite number, then the improved estimates above implies that we can extend the solution $v = (v_i)$ to a strictly larger (hyperbolic) time interval, which contradicts the definition of $s_1$ in \eqref{eq:def-s1}. Hence $s_1$ must be $+\infty$, which implies the global existence of the solution $v = (v_i)$ to the system \eqref{eq:model-mKG}.
\end{proof}


\section{Proof for the uniform pointwise decay result}\label{sec:decay}

\subsection{Decomposition and nonlinear transformation}

Our task in this section is to show the unified pointwise decay estimate \eqref{eq:m-decay}
$$
\big| v_i (t, x) \big| \lesssim {1\over t + m t^{3/2}},
$$
which corresponds to the usual wave decay and Klein-Gordon decay with $m = 0, 1$ respectively.

Recall the system \eqref{eq:model-mKG}
$$
\aligned
- \Box v_i + m^2 v_i 
&= N^{jk}_i Q_0 (v_j, v_k) + M^{jk \alpha \beta}_i Q_{\alpha \beta} (v_j, v_k),
\\
\big( v_i, \del_t v_i \big)(t_0)
&=
\big( v_{i0}, v_{i1} \big).
\endaligned
$$

In order to arrive at \eqref{eq:m-decay}, it suffices to show
\bel{eq:wave0}
\big| v_i (t, x) \big| \lesssim t^{-1},
\ee
which is because we already obtain 
$$
\big| v_i (t, x) \big| \lesssim m^{-1} t^{-3/2}
$$
in Section \ref{sec:existence}. To achieve \eqref{eq:wave0}, we first do a nonlinear transformation from $v_i$ to $V_i = v_i + N_i^{jk} v_j v_k$,
and then decompose $V_i$ into pieces
\be 
V_i = V_{c,i} + V_{m,i} + V_{n,i}.
\ee
For clarity, we note that we use $V_{c,i}$ to denote the decomposition with cubic nonlinearities, use $V_{m, i}$ to denote the decomposition with nonlinearities with $m$ dependent factors, and use $V_{n, i}$ to denote the decomposition with null nonlinearities of the type $Q_{\alpha\beta}$.
The functions $V$'s are solutions to the following (linear) Klein-Gordon equations:
\bel{eq:V-eq1}
\aligned
- \Box V_i + m^2 V_i 
=
& M_i^{jk\alpha\beta} Q_{\alpha\beta} (v_j, v_k)
- m^2 N_i^{jk} v_j v_k
\\
&+ N_i^{jk} v_k \big( N_j^{mn} Q_0 (v_m, v_n) + M_j^{mn\alpha\beta} Q_{\alpha\beta} (v_m, v_n) \big) 
\\
&+ N_i^{jk} v_j \big( N_k^{mn} Q_0(v_m, v_n) + M_k^{mn\alpha\beta} Q_{\alpha\beta} (v_m, v_n) \big),
\\
\big(V_i, \del_t V_i \big) (t_0) 
= 
&\big( V_{i0}, V_{i1}\big)
:=
\big(v_{i0} + N_i^{jk} v_{j0} v_{k0}, v_{i1} + N_i^{jk} (v_{j0} v_{k1} + v_{j1} v_{k0}) \big),
\endaligned
\ee

\bel{eq:v-eq2}
\aligned
-\Box V_{c, i} + m^2 V_{c, i}
=
&N_i^{jk} v_k \big( N_j^{mn} Q_0 (v_m, v_n) + M_j^{mn\alpha\beta} Q_{\alpha\beta} (v_m, v_n) \big) 
\\
&+ N_i^{jk} v_j \big( N_k^{mn} Q_0(v_m, v_n) + M_k^{mn\alpha\beta} Q_{\alpha\beta} (v_m, v_n) \big),
\\
\big(V_{c, i}, \del_t V_{c, i} \big) (t_0) 
=& \big( V_{i0}, V_{i1}\big),
\endaligned
\ee

\bel{eq:v-eq3}
\aligned
-\Box V_{m, i} + m^2 V_{m, i}
=
&- m^2 N_i^{jk} v_j v_k,
\\
\big(V_{m, i}, \del_t V_{m, i} \big) (t_0) 
=& (0, 0),
\endaligned
\ee

as well as

\bel{eq:v-eq4}
\aligned
- \Box V_{n, i} + m^2 V_{n, i}
=
& M_i^{jk\alpha\beta} Q_{\alpha\beta} (v_j, v_k),
\\
\big(V_{n, i}, \del_t V_{n, i} \big) (t_0) 
=& (0, 0).
\endaligned
\ee

In addition to the decomposition above, we find it helps to utilise the divergence structure of the null forms of the form $Q_{\alpha\beta}$ if we further decompose the $V_{n, i}$ part as
\be
V_{n,i} = V_{n,i}^0 + \del_\gamma V_{n,i}^{\gamma}.
\ee
We use $V_{n,i}^5, V_{n,i}^{\gamma}$ to denote the decomposition with 0 nonlinearities and divergent nonlinearities without $\del_\gamma$ respectively.
Similarly, $V_{n,i}^5, V_{n,i}^{\gamma}$ are solutions to the following (linear) Klein-Gordon equations:

\bel{eq:v-eq5}
\aligned
- \Box V_{n,i}^5 + m^2 V_{n,i}^5
&= 0,
\\
\big( V_{n,i}^5, \del_t V_{n,i}^5  \big) (t_0) 
&=
(0, 0),
\endaligned
\ee

\bel{eq:v-eq6}
\aligned
- \Box V_{n,i}^{\gamma} + m^2 V_{n,i}^{\gamma}
&= M_i^{jk\gamma\beta} v_j \del_\beta v_k - M_i^{jk\alpha\gamma} v_j \del_\alpha v_k,
\\
\big( V_{n,i}^{\gamma}, \del_t V_{n,i}^{\gamma} \big) (t_0)
&= (0, 0).
\endaligned
\ee

\subsection{The mass independent $L^2$ norm estimates}

Recall that our goal is to obtain the following mass independent $L^2$ estimates for functions $V$'s, and we will rely on the bootstrap method one more time to achieve it.

\begin{proposition}\label{prop:L2-m-Indpt}
For all $|I| + |J| \leq N-1$ (and for each $i$) we have
\bel{eq:L2-m-Indpt}
\aligned
\big\| \del^I L^J \big( V_{c, i}, V_{m, i}, V_{n, i}, V_{n, i}^5  \big) \big\|_{L^2(\RR^3)} 
+ t^{-\delta} \big\| \del^I L^J \del V_{n,i}^{\gamma} \big\|_{L^2(\RR^3)} 
\leq C_3 \eps,
\endaligned
\ee
and for all $|I| + |J| = N$ we have
\bel{eq:L2-m-Indpt}
\aligned
\big\| \del^I L^J \big( V_{c, i}, V_{m, i}, V_{n, i}^5  \big) \big\|_{L^2(\RR^3)} 
+ t^{-\delta} \big\| \del^I L^J \del V_{n,i}^{\gamma} \big\|_{L^2(\RR^3)} 
\leq C_3 \eps,
\endaligned
\ee
in which $C_3 \geq C_2 \geq C_1$ are some constants to be determined.
\end{proposition}

We have the following results, which is a consequence of the pointwise decay estimate 
$$
\big| \del^I L^J \del v, \del \del^I L^J v \big|
\lesssim C_1 \eps t^{-1},
\qquad
|I| + |J| \leq N-2,
$$  
obtained in Section \ref{sec:existence}.
\begin{lemma}
We have for all $|I| + |J| \leq N-2$ that
\bel{eq:EE-12} 
E_m (t, \del^I L^J v)^{1/2}
\lesssim 
C_1 \eps t^{\delta/2}.
\ee
\end{lemma}

Since we are proving energy estimates for linear equations, we know the solutions already exist. We first prove the energy estimates for the low order cases. By the continuity of the energies, we assume the following bounds are valid for all $t \in [t_0, t_1)$
\bel{eq:BA2}
\aligned
\big\| \del^I L^J \big( V_{c, i}, V_{m, i}, V_{n, i}, V_{n, i}^5  \big) \big\|_{L^2(\RR^3)} 
+ t^{-\delta} \big\| \del^I L^J \del V_{n,i}^{\gamma} \big\|_{L^2(\RR^3)} 
&\leq C_2 \eps,
\qquad
|I| + |J| \leq N-3,
\\
\big\| \del^I L^J \big( V_{c, i}, V_{m, i}, V_{n, i}^5  \big) \big\|_{L^2(\RR^3)} 
+ t^{-\delta} \big\| \del^I L^J \del V_{n,i}^{\gamma} \big\|_{L^2(\RR^3)} 
&\leq C_2 \eps,
\qquad
|I| + |J| \leq N-2.
\endaligned
\ee
Similar to the definition of $s_1$ in Section \ref{sec:existence}, $t_1$ is defined by
\bel{eq:def-t1}
t_1 := \sup \{ t : t > t_0, ~\eqref{eq:BA2}~ holds \}.
\ee

A direct result from the bootstrap assumption \eqref{eq:BA2} is the following.

\begin{lemma}\label{lem:Vtov1}
For all $t \in [t_0, t_1)$ we have
\be 
\aligned
&\big\| \del^I L^J v \big\|_{L^2(\RR^3)}
\lesssim C_2 \eps,
\qquad
|I| + |J| \leq N-3,
\\
&\big\| \del^I L^J v \big\|_{L^2(\RR^3)}
\lesssim C_2 \eps t^\delta,
\qquad
|I| + |J| = N-2.
\endaligned
\ee
\end{lemma}
\begin{proof}
For all $|I| + |J| \leq N-3$ we first have 
$$
\big\| \del^I L^J V \big\|_{L^2(\RR^3)}
\lesssim C_2 \eps,
$$
which simply follows from the relations
$$
V_i = V_{c,i} + V_{m,i} + V_{n,i},
\qquad
V_{n,i} = V_{n,i}^0 + \del_\gamma V_{n,i}^{\gamma},
$$
and the commutator estimate in Lemma \ref{lem:comm-est11}
$$
\big| [\del, L] u \big| \lesssim |\del u |.
$$

Next, we recall it holds that
$$
v_i 
= V_i - N_i^{jk} v_j v_k.
$$
Thus we get
$$
\aligned
&\sum_{|I| + |J| \leq N-3} \big\| \del^I L^J v_i \big\|_{L^2(\RR^3)}
\\
\lesssim
&\sum_{|I| + |J| \leq N-3} \big\| \del^I L^J V_i \big\|_{L^2(\RR^3)}
+
\sum_{|I| + |J| \leq N-3} \big\| \del^I L^J v_i \big\|_{L^2(\RR^3)} \sum_{|I| + |J| \leq  N-3} \big\| \del^I L^J v_i \big\|_{L^\infty(\RR^3)},
\endaligned
$$
and the pointwise estimate obtained in Lemma \ref{lem:h-est0} yields
$$
\sum_{|I| + |J| \leq N-3} \big\| \del^I L^J v_i \big\|_{L^2(\RR^3)}
\lesssim
\sum_{|I| + |J| \leq N-3} \big\| \del^I L^J V_i \big\|_{L^2(\RR^3)}
+
C_1 \eps t^{-1/2} \sum_{|I| + |J| \leq N-3} \big\| \del^I L^J v_i \big\|_{L^2(\RR^3)},
$$
and the smallness of $C_1 \eps$ allows us finally to obtain
$$
\sum_{|I| + |J| \leq N-3} \big\| \del^I L^J v_i \big\|_{L^2(\RR^3)}
\lesssim
\sum_{|I| + |J| \leq N-3} \big\| \del^I L^J V_i \big\|_{L^2(\RR^3)}.
$$

The bound for the case of $|I| + |J| = N-2$ can be derived in the same way, hence the proof is complete.
\end{proof}

We are going to derive the improved estimates under the bootstrap assumption of \eqref{eq:BA2}, and we first provide the improved estimates for low order cases.

\begin{proposition}\label{prop:low-refine}
Let the estimate in \eqref{eq:BA2} be true, then for any $t \in [t_0, t_1)$ we have
\be
\aligned
\big\| \del^I L^J \big( V_{c, i}, V_{m, i}, V_{n, i}, V_{n, i}^5  \big) \big\|_{L^2(\RR^3)} 
+ t^{-\delta} \big\| \del^I L^J \del V_{n,i}^{\gamma} \big\|_{L^2(\RR^3)} 
&\lesssim \eps + (C_2 \eps)^2,
\qquad
|I| + |J| \leq N-3,
\\
\big\| \del^I L^J \big( V_{c, i}, V_{m, i}, V_{n, i}^5  \big) \big\|_{L^2(\RR^3)} 
+ t^{-\delta} \big\| \del^I L^J \del V_{n,i}^{\gamma} \big\|_{L^2(\RR^3)} 
&\lesssim \eps + (C_2 \eps)^2,
\qquad
|I| + |J| \leq N-2.
\endaligned
\ee
\end{proposition}

\begin{proof}

In terms of the features of each term, we estimate them one by one.

\textbf{Estimate for $\big\| \del^I L^J V_{m, i} \big\|_{L^2(\RR^3)}$.}
We start with the easy one, and the energy estimate for the $\del^I L^J V_{m, i}$ equation gives
$$
E_m (t, \del^I L^J V_{m, i})^{1/2}
\leq
\int_{t_0}^t \big\| \del^I L^J F_{V_{m, i}} \big\|_{L^2(\RR^3)} \, dt',
$$
with
$$
F_{V_{m, i}} = - m^2 N_i^{jk} v_j v_k.
$$
Recall the $m$ dependent pointwise estimate
$$
m \big| \del^I L^J v \big| \lesssim C_1 \eps t^{-3/2},
$$
and we get
$$
\aligned
E_m (t, \del^I L^J V_{m, i})^{1/2}
\lesssim
&m \int_{t_0}^t \sum_{|I| + |J| \leq N-2} \big\| \del^I L^J v \big\|_{L^2(\RR^3)} \sum_{|I| + |J| \leq N-2} \big\| m \del^I L^J v \big\|_{L^\infty (\RR^3)}  \, dt'
\\
\lesssim
&m C_1 \eps C_2 \eps \int_{t_0}^t t'^{-3/2 + \delta} \, dt'
\lesssim  m C_1 \eps C_2 \eps.
\endaligned
$$
Hence the definition of the energy $E_m$ implies that
$$
\big\| \del^I L^J V_{m, i} \big\|_{L^2(\RR^3)} 
\lesssim C_1 \eps C_2 \eps.
$$

\textbf{Estimates for $\big\| \del^I L^J V_{c, i} \big\|_{L^2(\RR^3)}$, $\big\| \del^I L^J V_{n, i} \big\|_{L^2(\RR^3)}$, $\big\| \del^I L^J V_{n, i}^5 \big\|_{L^2(\RR^3)}$.}
Since the procedure is the same when estimating these three solutions, we gather the proof here.

For the equations with fast-decay nonlinearities, which is the situation now, we rely on Proposition \ref{prop:Va-key} to obtain the $m$ independent $L^2$ norm bounds. Thus it suffices to estimate $\| \text{nonlinearities} \|_{L^{6/5}(\RR^3)}$.

We find for estimating $\big\| \del^I L^J V_{c, i} \big\|_{L^2(\RR^3)}$ it suffices to show that
$$
\aligned
& \sum_{|I| + |J| \leq N-2} \Big\| \del^I L^J \Big( N_i^{jk} v_k \big( N_j^{mn} Q_0 (v_m, v_n) + M_j^{mn\alpha\beta} Q_{\alpha\beta} (v_m, v_n) \big) 
\\
&\hskip2cm + N_i^{jk} v_j \big( N_k^{mn} Q_0(v_m, v_n) + M_k^{mn\alpha\beta} Q_{\alpha\beta} (v_m, v_n) \big) \Big) \Big\|_{L^{6/5}(\RR^3)}
\\
\lesssim
& \sum_{|I| + |J| \leq N-2} \big\| \del^I L^J v \big\|_{L^2(\RR^3)} \sum_{|I| + |J| \leq N-2} \big\| \del^I L^J \del v \big\|_{L^6(\RR^3)}^2
\\
\lesssim
& C_2 \eps t^\delta \sum_{|I| + |J| \leq N-2} \big\| \del^I L^J \del v \big\|_{L^2}^{2/3} \sum_{|I| + |J| \leq N-2} \big\| \del^I L^J \del v \big\|_{L^\infty}^{4/3}
\\
\lesssim
& C_2 \eps (C_1 \eps)^2 t^{-4/3 + 2 \delta},
\endaligned
$$ 
in which we used the estimate \eqref{eq:EE-12}.

Similarly, in order to estimate $\big\| \del^I L^J V_{n, i} \big\|_{L^2(\RR^3)}$, we need to demonstrate that
$$
\aligned
& \sum_{|I| + |J| \leq N-3} \big\| \del^I L^J \big(M_i^{jk\alpha\beta} Q_{\alpha\beta} (v_j, v_k) \big) \|_{L^{6/5}(\RR^3)}
\\
\lesssim
& t^{-1} \sum_{|I| + |J| \leq N-2} \big\| \del^I L^J v \big\|_{L^2(\RR^3)} \sum_{|I| + |J| \leq N-3} \big\| \del \del^I L^J v \big\|_{L^3(\RR^3)}
\\
\lesssim
& C_2 \eps t^{-1 + \delta} \sum_{|I| + |J| \leq N-3} \big\| \del \del^I L^J v \big\|_{L^2(\RR^3)}^{2/3} \sum_{|I| + |J| \leq N-3} \big\| \del \del^I L^J v \big\|_{L^\infty(\RR^3)}^{1/3}
\\
\lesssim
& C_2 \eps C_1 \eps t^{-4/3 + \delta}.
\endaligned
$$

The estimate for $\big\| \del^I L^J V_{n, i}^5 \big\|_{L^2(\RR^3)}$ with $|I| + |J| \leq N-2$ is trivial according to Proposition \ref{prop:Va-key}, since the equation is homogeneous.

\textbf{Estimate for $\big\| \del^I L^J \del V_{n,i}^{\gamma}\big\|_{L^2(\RR^3)}$.}
We utilise the energy estimate for the equation of $\del^I L^J \del V_{n,i}^{\gamma}$ with $|I| + |J| \leq N-2$ to get
$$
\aligned
E_m (t, \del^I L^J V_{n,i}^{\gamma})^{1/2}
\leq 
&\int_{t_0}^t \big\| \del^I L^J \big( M_i^{jk\gamma\beta} v_j \del_\beta v_k - M_i^{jk\alpha\gamma} v_j \del_\alpha v_k \big) \big\|_{L^2(\RR^3)} \, dt'
\\
\lesssim
&\int_{t_0}^t \sum_{|I| + |J| \leq N-2} \big\| \del^I L^J v \big\|_{L^2(\RR^3)} \sum_{|I| + |J| \leq N-2} \big\| \del \del^I L^J v \big\|_{L^\infty(\RR^3)} \, dt'
\\
\lesssim
& \int_{t_0}^t C_2 \eps t'^\delta C_1 \eps t'^{-1} \, dt'
\lesssim C_2 \eps C_1 \eps t^\delta.
\endaligned
$$
By the definition of $E_m$ and the commutator estimates, we deduce
$$
\big\| \del^I L^J \del V_{n,i}^{\gamma}\big\|_{L^2(\RR^3)}
\lesssim C_2 \eps C_1 \eps t^\delta,
\qquad
|I| + |J| \leq N-2.
$$

\end{proof}

According to the refined bounds in Proposition \ref{prop:low-refine}, we easily know that the estimates in \eqref{eq:BA2} are true for all $t \in [t_0, + \infty)$ after carefully choosing $C_2$ large enough and $\eps$ sufficiently small (we might shrink the choice of $\eps$ in Proof \ref{proof:existence} if needed, and nothing else is affected). 

The Klainerman-Sobolev inequality \eqref{eq:Va-K-S2} together with \eqref{eq:BA2} provides the following mass independent results
\bel{eq:m-Idpdt-pointwise}
\aligned
&\big\| \del^I L^J v \big\|_{L^2(\RR^3)}
\lesssim C_2 \eps,
\qquad
|I| + |J| \leq N-3,
\\
&\big\| \del^I L^J v \big\|_{L^2(\RR^3)}
\lesssim C_2 \eps t^\delta,
\qquad
|I| + |J| = N-2,
\\
&\big\| \del^I L^J v \big\|_{L^\infty(\RR^3)}
\lesssim C_2 \eps t^{-1},
\qquad
|I| + |J| \leq N-6,
\endaligned
\ee
which are valid for all $t \in [t_0, +\infty)$.

Next, to proceed to prove the bounds of high order energies in Proposition \ref{prop:L2-m-Indpt}, we make new bootstrap assumptions for $t \in [t_0, t_2)$ (and for all $i$)
\bel{eq:BA3}
\aligned
\big\| \del^I L^J \big( V_{c, i}, V_{m, i}, V_{n, i}, V_{n, i}^5  \big) \big\|_{L^2(\RR^3)} 
+ t^{-\delta} \big\| \del^I L^J \del V_{n,i}^{\gamma} \big\|_{L^2(\RR^3)} 
&\leq C_3 \eps,
\qquad
|I| + |J| \leq N-1,
\\
\big\| \del^I L^J \big( V_{c, i}, V_{m, i}, V_{n, i}^5  \big) \big\|_{L^2(\RR^3)} 
+ t^{-\delta} \big\| \del^I L^J \del V_{n,i}^{\gamma} \big\|_{L^2(\RR^3)} 
&\leq C_3 \eps,
\qquad
|I| + |J| \leq N.
\endaligned
\ee
Similar to the definition of $t_1$, $t_2$ is defined by
\bel{eq:def-t2}
t_2 := \sup \{ t : t > t_0, ~\eqref{eq:BA3}~ holds \}.
\ee
Recall that $C_3 \geq C_2$ is to be determined.

We have the following refined estimates for high order cases.

\begin{proposition}\label{prop:high-refine}
Assuming the estimate in \eqref{eq:BA3} be true, then for all $t \in [t_0, t_2)$ we have
\be
\aligned
\big\| \del^I L^J \big( V_{c, i}, V_{m, i}, V_{n, i}, V_{n, i}^5  \big) \big\|_{L^2(\RR^3)} 
+ t^{-\delta} \big\| \del^I L^J \del V_{n,i}^{\gamma} \big\|_{L^2(\RR^3)} 
&\lesssim \eps + (C_3 \eps)^2,
\qquad
|I| + |J| \leq N-1,
\\
\big\| \del^I L^J \big( V_{c, i}, V_{m, i}, V_{n, i}^5  \big) \big\|_{L^2(\RR^3)} 
+ t^{-\delta} \big\| \del^I L^J \del V_{n,i}^{\gamma} \big\|_{L^2(\RR^3)} 
&\lesssim \eps + (C_3 \eps)^2,
\qquad
|I| + |J| \leq N.
\endaligned
\ee
\end{proposition}

The proof for Proposition \ref{prop:low-refine} also applies to Proposition \eqref{prop:high-refine}, so we omit it. Also, similarly, we can choose $C_3$ large enough, and $\eps$ sufficiently small (we shrink it further if needed), so that we can improve the estimates in \eqref{eq:BA3} with a factor $1/2$ in front of the original bounds. And this indicates that the estimates in \eqref{eq:BA3} are valid for all $t \in [t_0, + \infty)$.

The proof for Proposition \ref{prop:L2-m-Indpt} follows from the established estimates in \eqref{eq:BA2} and \eqref{eq:BA3}, which are valid for all $t \in [t_0, +\infty)$.

\subsection{The mass independent wave decay for the solution}

With the estimates built in Proposition \ref{prop:L2-m-Indpt}, we have the following results for the original solution $v = (v_i)$.

\begin{lemma}\label{lem:Vtov2}
For all $t \in [t_0, +\infty)$ we have
\be 
\aligned
&\big\| \del^I L^J v \big\|_{L^2(\RR^3)}
\lesssim C_3 \eps,
\qquad
|I| + |J| \leq N-1,
\\
&\big\| \del^I L^J v \big\|_{L^2(\RR^3)}
\lesssim C_3 \eps t^\delta,
\qquad
|I| + |J| = N.
\endaligned
\ee
\end{lemma}
The proof for Lemma \ref{lem:Vtov2} follows from the proof for Lemma \ref{lem:Vtov1}.

Next, we apply the Klainerman-Sobolev inequality \eqref{eq:Va-K-S2} to arrive at the mass independent pointwise decay results.

\begin{proposition}\label{prop:m-Indpt-p}
It holds that
\be 
\aligned
&\big\| \del^I L^J v \big\|_{L^\infty(\RR^3)}
\lesssim C_3 \eps t^{-1},
\qquad
|I| + |J| \leq N-4,
\\
&\big\| \del^I L^J v \big\|_{L^\infty(\RR^3)}
\lesssim C_3 \eps t^{-1 + \delta},
\qquad
|I| + |J| = N-3.
\endaligned
\ee
\end{proposition}

We now establish \eqref{eq:m-decay} in Theorem \ref{thm:main-mKG}.

\begin{proof}[Proof of \eqref{eq:m-decay}]
From Proposition \ref{prop:m-Indpt-p}, we obtain the mass independent pointwise decay result
$$
\big| v_i \big| 
\lesssim t^{-1},
$$
while the estimates in Lemma \ref{lem:BA1-consqt} give us the mass dependent Klein-Gordon decay
$$
\big| v_i \big| 
\lesssim m^{-1} t^{-3/2},
\qquad
m \in (0, 1].
$$
Combine these two kinds of decay bounds, we are led to 
\bel{eq:m-Indpdt-decay}
\big| v_i \big| 
\lesssim 
{1 \over t + m t^{3/2}},
\qquad
m \in (0, 1].
\ee
But we see \eqref{eq:m-Indpdt-decay} is obviously true for the case of $m = 0$ (because it is just the case of wave equations), and hence the proof for \eqref{eq:m-decay} is complete.

\end{proof}


\section{Proof for the covergence result}\label{sec:converge}

With the global existence result and the unified pointwise decay result prepared in the last two sections,
we now want to build the convergence result when the mass parameter $m$ goes to 0. 

\begin{proof}[Proof of Theorem \ref{thm:converge}]
We take the difference between the equation of $v_i^{(m)}$ and the equation of $v_i^{(0)}$ to have
$$
\aligned
&- \Box \big( v_i^{(m)} - v_i^{(0)} \big) + m^2 \big( v_i^{(m)} - v_i^{(0)} \big)
\\
=
& - m^2 v_i^{(0)}
+ N^{jk}_i Q_0 \big(v_j^{(m)} - v_j^{(0)}, v_k^{(m)} \big) 
+ N^{jk}_i Q_0 \big(v_j^{(0)}, v_k^{(m)} - v_k^{(0)} \big) 
\\
&+ M^{jk \alpha \beta}_i Q_{\alpha \beta} \big(v_j^{(m)} - v_j^{(0)}, v_k^{(m)} \big) 
+ M^{jk \alpha \beta}_i Q_{\alpha \beta} \big(v_j^{(0)}, v_k^{(m)} - v_k^{(0)} \big) =: F^{(m)},
\endaligned
$$
with zero initial data
$$
\big( v_i^{(m)} - v_i^{(0)}, \del_t v_i^{(m)} - \del_t v_i^{(0)} \big) (t_0) = (0, 0).
$$

Acting the vector field $\del^I L^J$ with $|I| + |J| \leq N-2$ to the equation of $v_i^{(m)} - v_i^{(0)}$, and then the energy estimates give
$$
\aligned
E_m \big(t, \del^I L^J v_i^{(m)} - \del^I L^J v_i^{(0)} \big)^{1/2}
\leq 
&\int_{t_0}^t \big\| \del^I L^J F^{(m)} \big\|_{L^2(\RR^3)} \, dx,
\endaligned
$$
and by the fact $\big| \del \del^I L^J v \big| \lesssim C_1 \eps t^{-1}$ we further have 
$$
\aligned
E_m \big(t, \del^I L^J v_i^{(m)} - \del^I L^J v_i^{(0)} \big)^{1/2}
\lesssim
m^2 C_1 \eps t 
+
C_1 \eps \sum_{|I| + |J| \leq N-2} \int_{t_0}^t t'^{-1} E_m \big(t', \del^I L^J v_i^{(m)} - \del^I L^J v_i^{(0)} \big)^{1/2} \, dt',
\endaligned
$$
In succession, we apply the Gronwall's inequality to obtain
\bel{eq:convergence12}
\sum_{|I| + |J| \leq N-2} E_m \big(t, \del^I L^J v_i^{(m)} - \del^I L^J v_i^{(0)} \big)^{1/2}
\lesssim
m^2 C_1 \eps t^{1 + C C_1 \eps},
\ee
with $C$ a generic constant.

Thus by choosing $\eps$ sufficiently small such that $C C_1 \eps \leq \delta/2$, we complete the proof for the cases of $|I| + |J| \leq N-2$.

Based on the estimates obtained, and the Klainerman-Sobolev inequality \eqref{eq:Va-K-S2}, we obtain
$$
\sum_{{|I| + |J| \leq N-5}} \big\| \del \del^I L^J v_i^{(m)} - \del \del^I L^J v_i^{(0)} \big\|_{L^\infty(\RR^3)}
\lesssim 
m^2 C_1 \eps t^{-1},
$$
then for $|I| + |J| \leq N$ with $N\geq 6$ we can bound
$$
\int_{t_0}^t \big\| \del^I L^J F^{(m)} \big\|_{L^2(\RR^3)} \, dx
\lesssim
m^2 C_1 \eps t^{1+\delta/2} 
+
C_1 \eps \sum_{|I| + |J| \leq N} \int_{t_0}^t t'^{-1} E_m \big(t', \del^I L^J v_i^{(m)} - \del^I L^J v_i^{(0)} \big)^{1/2} \, dt',
$$
which further yields
$$
\aligned
&\sum_{|I| + |J| \leq N} E_m \big(t, \del^I L^J v_i^{(m)} - \del^I L^J v_i^{(0)} \big)^{1/2}
\\
\lesssim
& m^2 C_1 \eps t^{1+\delta/2} 
+
C_1 \eps \sum_{|I| + |J| \leq N} \int_{t_0}^t t'^{-1} E_m \big(t', \del^I L^J v_i^{(m)} - \del^I L^J v_i^{(0)} \big)^{1/2} \, dt'.
\endaligned
$$
Again, the Gronwall's inequality deduces that (by letting $\eps$ sufficiently small)
\bel{eq:convergence13}
\sum_{|I| + |J| \leq N} E_m \big(t, \del^I L^J v_i^{(m)} - \del^I L^J v_i^{(0)} \big)^{1/2}
\lesssim
m^2 C_1 \eps t^{1 + \delta}.
\ee

Till now, the proof is complete.
\end{proof}


\section*{Acknowledgements}

The author is grateful to Philippe G. LeFloch (Sorbonne University) for introducing this problem to him, and for many discussions. The author also would like  to thank Siyuan Ma (Sorbonne University) and Chenmin Sun (University of Cergy-Pontoise) for helpful discussions. 



\end{document}